\documentclass[10pt]{article}
\usepackage[T2A]{fontenc}
\usepackage[utf8]{inputenc}
\usepackage{amsmath, amsthm}\usepackage{enumerate}
\usepackage{amssymb}\usepackage{bbold}
\usepackage{color}
\renewcommand{\subsection}{\@startsection{subsection}{1}{0pt}{3.25ex plus 1ex minus.2ex}{-1em}{\normalfont\normalsize\bf}}

\newtheorem{theorem}{Theorem}[section]
\newtheorem{proposition}{Proposition}[section]
\newtheorem{lemma}{Lemma}[section]
\newtheorem{example}{Example}[section]
\newtheorem{corollary}{Corollary}[section]
\newtheorem{assertion}{Assertion}[section]
\newtheorem{definition}{Definition}[section]
\DeclareMathOperator{\convr}{\xrightarrow[]{{\it ru}}}
\DeclareMathOperator{\convtau}{\xrightarrow[]{\tau}}
\DeclareMathOperator{\convtaur}{\xrightarrow[]{\tau_{{\it ru}}}}

\newcommand{\spa}{\hbox{\rm span}}

\newcommand{\card}{\hbox{\rm card}}
\renewcommand{\a}{\alpha}
\renewcommand{\b}{\beta}
\newcommand{\g}{\gamma}
\newcommand{\w}{\omega}

\large
\begin{document}

\date{\empty}
\title{{\bf Relative uniform convergence and Archimedean property in pre-ordered vector spaces and Abelian groups}}
\maketitle
\author{\centering{{I. A. Emelyanenkov$^{1}$, E. Yu. Emelyanov$^{1}$\\ 
\small $1$ Sobolev Institute of Mathematics, Novosibirsk, Russia}
\abstract{It is proved that, for a pre-ordered vector space $X$, the quotient space $(X/A,[W])$ is an Archimedeanization of $X$,
where $W$ is the closure of positive wedge $X_+$ in ru-topology, $A=W\cap-W$, and $[W]$ is the quotient set of $W$ in $X/A$.
A construction of an almost Archimedeanization of a pre-ordered vector space is given. Analogous results are presented 
for pre-ordered groups. We present an explicit example of a family of vector lattices for which a chain of spaces of infinitely small elements
stabilizes first at any chosen ordinal ${\a}\leqslant{\w}_1$.}

\vspace{3mm}

\vspace{3mm}
}}

\section{Introduction and preliminaries}

\hspace{4mm}
Archimedean property is important in the functional repre\-sentation of various analytic and algebraic structures.
For example, Kadison's representation theorem \cite{Kadison-1951} tells that every ordered real vector space with an Archimedean
order unit is order isomorphic to a vector subspace of the space of continuous real valued functions on a compact Hausdorff space. 
Kadison's theorem inspired M.D. Choi and E.G. Effros to obtain an analogous representation theorem 
for self-adjoint subspaces of unital $C^\ast$-alge\-bras that contain the unit \cite{CE-1977}.
There is a plenty of recent works where the Archimedean property is studied in connection to various applica\-tions in algebra, analysis and quantum physics. 
For instance, in the breakthrough paper \cite{PT-2009}, V.I. Paulsen and M. Tomforde have developed a theory of ordered *-vector spaces.
In their theory, Archimed\-eanization of ordered vector spaces with order units plays a crucial role. 
The const\-ruc\-tion of Archimedeanization of ordered vector spaces with an order unit \cite{PT-2009} was extended to 
arbitrary ordered vector spaces in \cite{E-2017}. The proof of existence of an Archimedeanization given in \cite{E-2017} has a gap, 
that is fixed in the present paper. Let us also mention one more important direction in analysis 
based on represen\-tation theorems for Archimedian structu\-res, the const\-ruction of free spaces with additional analytic structures that 
was initiated by B. de Pagter and A.W. Wickstead in \cite{PW-2015}. In what follows, vector spaces are real. We recall several definition.

\begin{definition}\label{Definition-1_1}
A subset $A$ of a vector space {\em (}of a group{\em )} $X$ is {\color{blue}{generating}}, whenever $A-A=X$.
\end{definition}

\begin{definition}\label{Definition-1_2}
A non-empty subset $W$ of a vector space $X$ is a {\color{blue}{wedge}} whenever 
$W+W\subseteq W \ \ \text{\rm and} \ \ tW\subseteq W \ \ \text{\rm for} \ \ t\geqslant 0$.
A wedge $C$ is called a {\color{blue}{cone}}, if $C\cap -C=\{0\}$.
\end{definition}

\noindent
Under an {\color{blue}{pre-ordered vector space}} (briefly, {\color{blue}{POVS}})
we understand a vector space $X$ together with a wedge $W$. 
When $W$ is a cone, we say that $X=(X,W)$ is an {\color{blue}{ordered vector space}} (briefly, {\color{blue}{OVS}}).

\begin{definition}\label{Definition-1_3}
A subset $A$ of a POVS $(X,W)$ is {\color{blue}{majorizing}} if for each $x\in X$ there exists $a\in A$ such that $a-x\in W$.
\end{definition}

\noindent
It can be seen easily that the positive wedge $W$ of a POVS $X$ is majorizing iff $W$ is generating iff $\spa(W)=X$.

\begin{definition}\label{Definition-1_4}
Every POVS $(X,W)$ is equipped with a {\color{blue}{partial pre-order}}
$$
   x\leqslant_W y\ \ \text{\rm(or, \ simply}  \ \ x\leqslant y)\ \Longleftrightarrow\ y-x\in W.
$$ 
Each pair of vectors $\{a,b\}$ in a POVS $(X,W)$ produces a {\em (}possibly empty{\em)} {\color{blue}{order interval}}:
$$
   [a,b]_W=[a,b]:=\{x\in X: a\leqslant_W x\leqslant_W b\}=(a+W)\cap(b-W).
$$
\end{definition}

\begin{definition}\label{Definition-1_5}
A subset $A$ of a POVS $(X,W)$ is an {\color{blue}{order ideal}} in $(X,W)$, whenever $A$ is a {\color{blue}{vector subspace}} of $X$ such that 
$a,b\in A\Longrightarrow [a,b]\subseteq A$. 
\end{definition}

\noindent
If $A$ is an order ideal in a POVS $(X,W)$, the quotient space $X/A$ is pre-ordered by:
$$
   [0] \leqslant [x] \text{ if there exists}\ y \in A \text{ such that } 0 \leqslant_W x + y.
$$

\begin{definition}\label{Definition-1_6}
A wedge $W$ in a vector space $X$ {\em (}or, a POVS $(X,W)${\em )} is:
\begin{enumerate}[-]
\item\
{\color{blue}{almost Archimedean}} if $y=0$, whenever $\pm ny\leqslant_W u\in W$ for all $n\in\mathbb{N}$.
\item\
{\color{blue}{Archimedean}} if $y\leqslant_W 0$, whenever $ny\leqslant_W u\in W$ for all $n\in\mathbb{N}$.
\end{enumerate}
\end{definition}

\noindent
It is straightforward to see that every almost Archimedean wedge is a cone, every subcone of an almost Archimedean cone is almost Archimedean, and
every Archimedean cone is almost Archime\-dean. A wedge $W=X$ in a vector space $X$ is Archimedean yet not almost Archimedean unless $X=\{0\}$.
Even a two-dimensional almost Archimedean OVS need not to be Archimedean
(consider $\mathbb{R}^2$, ordered by a cone $C=\{(r_1,r_2):\ \text{\rm either}\ r_1=r_2=0, \ \text{\rm or}\ \min(r_1,r_2)>0\}$).
Furthermore, a POVS $(X,W)$ is almost Archime\-dean if and only if $W$ does not contain a straight line (see, e.g., Proposition \ref{Proposition-3_1}~$iv)$).

\medskip
It is convenient to denote the wedge $W$ of a POVS (X,W) by $X_+$, call it a {\color{blue}{positive wedge}} of $X$, 
and write $x\leqslant y$ instead of $x\leqslant_{X_+} y$.

\begin{definition}\label{Definition-1_7}
An operator $T: X\to Y$ between POVSs $X$ and $Y$ is {\color{blue}{positive}} {\em (}{\color{blue}{order bounded}}{\em )} if $T$ takes $X_+$ into $Y_+$
{\em (}order intervals of $X$ into order intervals of $Y${\em )}.
\end{definition}

\begin{definition}\label{Definition-1_8}
A net $(x_{\a})$ in a POVS $X$ 
{\color{blue}{ru-converges}} to $x\in X$ {\em (}shortly, $x_{\a}\convr x${\em )} 
if, for some {\em (}{\color{blue}{regulator of the convergence}}{\em )} $w\in X$, there is a sequence 
$({\a}_n)$ of indices with $\pm n(x_{\a}-x)\leqslant w$ for ${\a}\geqslant{\a}_n$.
\end{definition}

\noindent
Clearly, a regulator $w$ belongs to $X_+$. Whenever it is convenient we write $x_{\a}\convr x(w)$. 

\begin{definition}\label{Definition-1_9}
A subset $S$ of a POVS $X$ is {\color{blue}{ru-closed}} if, for every net $x_{\a}$ in $S$ and every $w\in X_+$ it follows from $x_{\a}\convr x(w)$ that $x\in S$.
\end{definition}

\noindent
The notion of ru-convergence was introduced by L.V. Kantorovich in \cite{Kantorovich-1936} as
an abstraction of the classical uniform conver\-gence of functions in $C[0,1]$.
It is straightforward to see that every ru-convergent net in an OVS $X$ has a unique ru-limit if and only if $X$ is almost Archime\-dean.
As the ru-con\-vergence is sequential, we can always restrict ourselves to ru-convergent sequences.

\medskip
The present paper is devoted to interrelation between the Archime\-dean property and relatively uniform conver\-gence in pre-ordered vector spaces.
In Theorem \ref{Theorem-2_1}, we prove that the quotient set $[W]$ of the $\tau_{ru}$-closure of positive wedge $X_+$ of a pre-ordered vector space $X$
is an Archime\-dean cone in quotient space $X/(W\cap -W)$. One of main results of our paper is Theorem \ref{Theorem-2_2} which establishes
that $(X/(W\cap -W),[W])$ is an Archimedeanization of $X$. Another line of research in our paper is concerned to an almost Archimedeanization of a POVS.
In Section 4 we present examples of OVSs and of vector lattices whose almost Archimedeanization type is an arbitrary ordinal
$\a$ such that $\a\leqslant\w_1$. Finally, in Section 5 contains some versions of results of Sections 2 and 3 for Abelian pre-ordered groups.

\smallskip
For further unexplained terminology and notations we refer to \cite{AT-2007,EEG-2025,KG-2019,KM-1994}.

\section{ru-Closedness of the positive wedge vs the Archi\-medean property}

\hspace{4mm}
This section is aimed on an investigation how the ru-closedness of a wedge interplay with the Archi\-medean property.
In the end of the section, we give a new construction of an Archimedeanization of a POVS  in Theorem \ref{Theorem-2_2}.

\smallskip
Let $X/A$ be a quotient space of a set $X$ by subset $A$ of $X$. A quotient set of $B\subseteq X$ is the set $[B]=\{[x]\in X/A: x\in B\}$
of all equivalence classes formed from $B$ in $X/A$. The following lemma is a standard fact.

\begin{lemma}\label{Lemma-2_1}
Let $W$ be a wedge in a vector space $X$ and $A=W\cap -W$. The following assertions hold.
\begin{enumerate}[$i)$]
\item\
$A$ is an order ideal in the POVS $(X,W)$.
\item\
$[W]$ is a cone in $X/A$.
\item\
$[W]$ is generating in $X/A$ if and only if $W$ is generating in $X$.
\end{enumerate}
\end{lemma}

\begin{proof}
$i)$\
A straightforward computation shows that $A$ is a subspace of $X$.
Let $a,b\in A$ and $a\leqslant_W x\leqslant_W b$ for some $x\in X$. Since $\pm a,\pm b\in W$ and $b-x,x-a\in W$,
then $x=a+x-a\in W$ and $-x=-b+b-x\in W$, and hence $x\in A$. Thus, $A$ is an order ideal in $(X,W)$.

\medskip
$ii)$\
Clearly, $[W]$ is a wedge in $X/A$. To prove $[W]$ is a cone it suffices to show $[W]\cap-[W]\subseteq A$.
Let $[x]\in[W]\cap-[W]$. Then $x\in x_1+A$ and $x\in -x_2+A$ for some $x_1,x_2\in W$.
It follows $x\in(W+W)\cap(-W-W)\subseteq W\cap -W=A$, and hence $[x]=0$. 

\medskip
$iii)$\
Suppose $W-W=X$ and let $[x]\in X/A$. Then, $x=y_1-y_2$ for some $y_1,y_2\in W$, and hence
$[x]=[y_1-y_2]=[y_1]-[y_2]\in[W]-[W]$. Thus, $[W]$ is generating in $X/A$.

\smallskip
Suppose $[W]-[W]=X/A$ and let $x\in X$. Take $y_1,y_2\in W$ with $[x]=[y_1]-[y_2]=[y_1-y_2]$. 
So, for some $z\in A$ we have $x+z=y_1-y_2$, and hence $x=y_1-(y_2+z)\in W-W$. Therefore, the wedge $W$ is generating in $X$.
\end{proof}

An OVS $X$ is a {\color{blue}{vector lattice}} if $|x|:=\sup\{x,-x\}$ exists for every $x\in X$. It is easy to see that
every almost Archimedean vector lattice is Archimedean. An order ideal $A$ in a vector lattice $X$ will be referred to as a {\color{blue}{lattice ideal}}
whenever it  satisfies the property: $a\in A\Longrightarrow|a|\in A$.

\smallskip
W.A.J. Luxemburg and L.C-Jr. Moore proved in \cite{LM-1967} that the ru-closed subsets of a vector lattice $X$ are exactly the closed sets 
of the so-called ru-topology on $X$. Such a topology can be also introduced on an arbitrary POVS \cite{KG-2019}.

\begin{definition}\label{Definition-2_1}
The {\color{blue}{ru-topology}}, denoted by $\tau_{ru}$, on a POVS $X$ is determined as follows: a subset $S$ of $X$
is $\tau_{ru}$-closed whenever $S\ni x_n\convr x$ implies $x\in S$. The {\color{blue}{$\tau_{ru}$-closure}} of a subset $S$ of $X$ is denoted by $\overline{S}_{ru}$.
\end{definition}

It follows immediately from Definition \ref{Definition-2_1} that the ru-topology on a POVS $X$ is the strongest topology
$\tau$ on $X$ with the property $x_n\convr x\Longrightarrow x_n\convtau x$. Clearly, $\tau_{ru}$ is translation invariant.
In general, the topology $\tau_{ru}$ is not linear (e.g., on each POVS with a non-generating positive wedge).

It is useful to describe more constructi\-vely the $\tau_{ru}$-closure of a subset of a POVS.
The key idea is coming from \cite{LM-1967}, where the pseudo uniform closure of a subset of a vector lattice was studied.
Adopting the idea to the POVS setting, for a subset $S$ of a POVS $X$ we denote $S^{(0)}_{ru}:=S$;
$S^{(1)}_{ru}:=S'_{ru}$ the set of all $x\in X$ such that $x_n\convr x$ for some sequence $(x_n)_n$ in $S$;
$S^{({\a}+1)}_{ru}:=\bigl(S^{({\a})}_{ru}\bigl)'_{ru}$; and
$S^{({\a})}_{ru}:=\bigcup\limits_{\b\in{\a}}S^{(\b)}_{ru}$ when ${\a}$ is limit ordinal.

\begin{lemma}\label{Lemma-2_2}
Let $S$ be a subset of a POVS $X$. The following assertions hold.
\begin{enumerate}[$i)$]
\item\
$S^{({\a}_1)}_{ru}\subseteq S^{({\a}_2)}_{ru}$, whenever ${\a}_1\leqslant{\a}_2$.
\item\
$\overline{S}_{ru}=\bigcup\limits_{{\a}\in{\w}_1}S^{({\a})}_{ru}$, where ${\w}_1$ is the first uncountable ordinal.
\item\
$S\subseteq X_+-X_+\Longleftrightarrow\overline{S}_{ru}\subseteq X_+-X_+$.
\end{enumerate}
\end{lemma}

\begin{proof}
$i)$\
It is straightforward.

\medskip
$ii)$\
It is similar to the proof of \cite[Theorem 3.3]{LM-1967}, where the vector lattice case is considered.

\medskip
$iii)$\
It suffices to prove $S\subseteq X_+-X_+\Longrightarrow S'_{ru}\subseteq X_+-X_+$. So, let $S\subseteq X_+-X_+$
and $S\ni x_n\convr x$. By passing to a subsequence, if necessary, we may assume that 
$\pm n(x_n-x)\leqslant w$ for some $w\in X_+$ and all $n\in\mathbb{N}$. 
In particular, $x-x_1\leqslant w$, or $x\leqslant w+x_1$. Since $x_1=x_1^1-x_1^2$ for some $x_1^1,x_1^2\in X_+$, we obtain $x\leqslant w+x_1^1$ and 
$x=(w+x_1^1)-(w+x_1^1-x)\in X_+-X_+$.
\end{proof}

\begin{lemma}\label{Lemma-2_3}
Let $W$ be a wedge in a POVS $X$. The following assertions hold.
\begin{enumerate}[$i)$]
\item\
$W^{({\a})}_{ru}$ is a wedge in $X$ for every ordinal ${\a}$.
\item\
$\overline{W}_{ru}=W^{({\w}_1)}_{ru}$ is a wedge in $X$.
\end{enumerate}
\end{lemma}

\begin{proof}
$i)$\
Let ${\a}$ be an ordinal. 
We have to prove that $t\cdot W^{({\a})}_{ru}\subseteq W^{({\a})}_{ru}$ for all $t\in\mathbb{R}_+$, and that
$W^{({\a})}_{ru}+W^{({\a})}_{ru}\subseteq W^{({\a})}_{ru}$.

\smallskip
By definition, $W^{(0)}_{ru}=W$, so $W^{(0)}_{ru}$ is a wedge. 
Suppose $W^{(\b)}_{ru}$ is a wedge for every $\b<{\a}$.
First, consider the case when ${\a}=\b+1$ for some $\b$.
So, $W^{(\b)}_{ru}$ is a wedge and $W^{({\a})}_{ru}=\bigl(W^{(\b)}_{ru}\bigl)'_{ru}$.
Let $t\geqslant 0$ and $x\in W^{({\a})}_{ru}$. Find a sequence $(x_n)$ in $W^{(\b)}_{ru}$ with $x_n\convr x$.
By passing to a subsequence, we may assume that $\pm n(x_n-x)\leqslant w$ for some $w\in X_+$ and all $n\in\mathbb{N}$.
Then, $\pm n(tx_n-tx)\leqslant tw$ for all $n\in\mathbb{N}$. We conclude $W^{(\b)}_{ru}\ni tx_n\convr tx$, and hence 
$tx\in\bigl(W^{(\b)}_{ru}\bigl)'_{ru}=W^{({\a})}_{ru}$.
Now, let ${\a}$ be a limit ordinal, $t\geqslant 0$, and $x\in W^{({\a})}_{ru}=\bigcup\limits_{\b\in{\a}}W^{(\b)}_{ru}$.
Then, $x\in W^{(\b)}_{ru}$ for some $\b<{\a}$. By the assumption, $tx\in tW^{(\b)}_{ru}\subseteq W^{(\b)}_{ru}\subseteq W^{({\a})}_{ru}$.
Again, we conclude $tx\in W^{({\a})}_{ru}$.

\smallskip
The proof of $W^{({\a})}_{ru}+W^{({\a})}_{ru}\subseteq W^{({\a})}_{ru}$ is similar, we point out only the case of 
${\a}=\b+1$. So, let $x,y\in W^{(\b+1)}_{ru}$, say $W^{(\b)}_{ru}\ni x_n\convr x$ and $W^{(\b)}_{ru}\ni y_n\convr y$.
By passing to subsequences, we may assume $\pm n(x_n-x)\leqslant w$ and $\pm n(y_n-y)\leqslant v$ for some $w,v\in X_+$ and all $n\in\mathbb{N}$.
Then, $\pm n((x_n+y_n)-(x+y))\leqslant (w+v)$ for all $n\in\mathbb{N}$. As $W^{(\b)}_{ru}$ is a wedge, 
we conclude $W^{(\b)}_{ru}\ni (x_n+y_n)\convr (x+y)$, and hence $x+y\in\bigl(W^{(\b)}_{ru}\bigl)'_{ru}=W^{(\b+1)}_{ru}$.

\medskip
$ii)$\
It follows from $i)$ due to Lemma \ref{Lemma-2_2}~$ii)$.
\end{proof}

It is proved in \cite[Theorem 4.2]{LM-1967} that every lattice homomorphism between vector lattices is continuous
in the ru-topology. Indeed, it is true for an arbitrary positive linear operator between POVSs and, if additionally the positive wedge in the co-domain is generating, 
it is also true for every order bounded linear operator.

\begin{lemma}\label{Lemma-2_4}
Let $T:X\to Y$ a linear operator between POVSs. Then $T$ is $\tau_{ru}$-continuous in each of the following cases.
\begin{enumerate}[$i)$]
\item\
$T$ is positive.
\item\
The positive wedge $Y_+$ is generating and $T$ is order bounded.
\end{enumerate}
\end{lemma}

\begin{proof}
Let $S\subseteq Y$ be $\tau_{ru}$-closed. In order to show $T$ is $\tau_{ru}$-continuous, it suffices to prove that $T^{-1}(S)$ is $\tau_{ru}$-closed.
To this end, let $T^{-1}(S)\ni x_n\convr x$.  By passing to a subsequence, we may suppose 
$\pm n(x_n-x)\leqslant u$ for some $u\in X_+$ and all $n\in\mathbb{N}$. By definition of $\tau_{ru}$-topology,
it remains to show $x\in T^{-1}(S)$.

\smallskip
$i)$\
Let $T\geqslant 0$. Then $\pm n(Tx_n-Tx)\leqslant Tu$ for all $n\in\mathbb{N}$, and hence $Tx_n\convr Tx$. Thus,
$S\ni Tx_n\convtaur Tx$ and, since $S$ is $\tau_{ru}$-closed, we infer $Tx\in S$, or $x\in T^{-1}(S)$.

\smallskip
$ii)$\
Let $Y_+$ be generating and $T$ order bounded. Then $T$ is ru-continuous by \cite[Theorem 2.4~ii)]{EEG-2025}.
Thus, $T^{-1}(S)\ni x_n\convr x$ implies $S\ni Tx_n\convr Tx$. So, $S\ni Tx_n\convtaur Tx$. 
Since $S$ is $\tau_{ru}$-closed, $x\in T^{-1}(S)$.
\end{proof}

It was observed by T. Ito \cite{Ito-1967} that a vector lattice $X$ is Archimedean if and only if $X_+$ is ru-closed. 
The following elementary lemma shows that it is happened also in every POVS.

\begin{lemma}\label{Lemma-2_5}
A POVS $X$ is Archimedean if and only if $X_+$ is ru-closed.
\end{lemma}

\begin{proof}
$\Longrightarrow$:\
Let $X$ be Archimedean and $X_+\ni x_n\convr x_0(w)$. By passing to a subsequence, we may suppose $\pm n(x_0-x_n)\leqslant w$ for all $n\in\mathbb{N}$. Then, 
$nx_0=n(x_0-x_n)+nx_n\geqslant n(x_0-x_n)\geqslant-w$, and hence $w\geqslant n(-x_0)$ for all $n\in\mathbb{N}$. Since $X$  is Archimedean, $(-x_0)\leqslant 0$, or $x_0\in X_+$. 

$\Longleftarrow$:\
Let $X_+$ be ru-closed. Assume $nx\leqslant y\in X_+$ for all $n\in\mathbb{N}$ and some $x\in X$.
Thus, $\frac{1}{n}y-x\in X_+$ for all $n\in\mathbb{N}$. Since $\frac{1}{n}y\convr 0$  
then $\frac{1}{n}y-x\convr-x$. As $X_+$ is ru-closed, it follows $-x\in X_+$ or $x\leqslant 0$. Therefore, $X$ is Archimedean.
\end{proof}

\noindent
Lemma \ref{Lemma-2_5} should be compared with the following classical fact (cf., \cite[Lemma 2.3]{AT-2007}). 

\begin{assertion}\label{Assertion-2_1}
Every POVS $X$ admitting a linear topology $\tau$, for which $X_+$ is $\tau$-closed, is Archimedean.
\end{assertion}

\noindent
It is worth mentioning a recent paper \cite{GE-2023} by A.E. Gutman and the first author,
where an exhaustive description of the class of locally convex spaces in which all Archimedean cones are topologically closed is obtained.
However $\tau_{ru}$-topology need not to be linear. But, as it was men\-tion\-ed above,
$\tau_{ru}$-topology on a POVS $X$ is the strongest topology $\tau$ on $X$ such that $x_n\convr x\Longrightarrow x_n\convtau x$.
This leads to the following proposition.

\begin{proposition}\label{Proposition-2_1}
Let $X$ be a POVS. The following assertions are equivalent.
\begin{enumerate}[$i)$]
\item\
$X$ is Archimedean.
\item\
$X_+$ is ru-closed.
\item\
$(X_+)'_{ru}=X_+$.
\item\
$X_+$ is closed in the topology $\tau_{ru}$.
\item\
$X_+$ is closed in some topology $\tau$ on $X$ possessing the property: $x_n\convr x\Longrightarrow x_n\convtau x$.
\end{enumerate}
\end{proposition}

\begin{proof}
$i)\iff ii)$ is the statement of Lemma \ref{Lemma-2_5}, $ii)\Longrightarrow iv)$ follows from Definition \ref{Definition-2_1}, and $ii)\iff iii)$ and $iv)\Longrightarrow v)$ are trivial.

\medskip
$v)\Longrightarrow ii)$\
Let $X_+$ be closed in a topology $\tau$ on $X$ satisfying $x_n\convr x\Longrightarrow x_n\convtau x$.
If $X_+$ is not ru-closed, for some sequence $(x_n)$ we have $X_+\ni x_n\convr x\notin X_+$.
Since $X_+$ is $\tau$-closed and $X_+\ni x_n\convtau x$, we obtain $x\in X_+$, a contradiction.
\end{proof}

One of the most exciting results concerning interrelations between Archimedean 
cones and ru-convergence is the Luxemburg-Moore-Veksler theorem:

\begin{assertion}\label{Assertion-2_2}
{\em (\cite{LM-1967,Veksler-1966})}
A lattice ideal $A$ in a vector lattice $X$ is ru-closed if and only $X/A$ is Archimedean. 
\end{assertion}

\noindent
The forward implication of Assertion \ref{Assertion-2_2} is rather trivial in the OVS case \cite[Proposition 3.2(a)]{E-2017}.
On the other hand, for an OVS $V$ from \cite[Example 2.5]{E-2017}, $V = V/\{0\}$ is not Archimedean although the ideal $\{0\}$ is ru-closed.
We do not have yet an extension of Assertion \ref{Assertion-2_2} to reasonable classes of POVSs beside vector lattices. 
However, for some natural ru-closed order ideal $A$ in a POVS $X$, the space $X/A$ equipped with a cone that is the quotient of the $\tau_{ru}$-closure of $X_+$ 
becomes an Archimedean OVS.

\begin{theorem}\label{Theorem-2_1}
Let $X$ be a POVS, $W$ is $\tau_{ru}$-closure of $X_+$, and $A=W\cap -W$. 
The following assertions hold.
\begin{enumerate}[$i)$]
\item\
$[W]$ is an Archime\-dean cone in $X/A$. 
\item\
$[X_+]$ is an almost Archime\-dean cone in $X/A$. 
\item\
$[W]$ is generating if and only if $X_+$ is generating.
\end{enumerate}
\end{theorem}

\begin{proof}
$i)$\
By Lemma \ref{Lemma-2_3}, $W$ is a wedge.
Lemma \ref{Lemma-2_1} implies that $[W]$ is a cone in $X/A$. 
Let $n[x]\leqslant_{[W]} [y]$ for some $x\in X$, $y\in W$, and all $n\in\mathbb{N}$.
Then $[x]\leqslant_{[W]}\frac{1}{n}[y]$, and hence $[\frac{1}{n}y-x]\in[W]$ for all $n\in\mathbb{N}$.
Therefore, there exists a sequence $(a_n)_n$ in $A$ with $\frac{1}{n}y-x+a_n\in W$.
From $A=W\cap -W$ we infer $-a_n\in W$ and hence $\frac{1}{n}y-x=-a_n+\frac{1}{n}y-x+a_n\in W+W\subseteq W$
for every $n\in\mathbb{N}$.
Since $X_+\subseteq X_+-X_+$ then $W=\overline{(X_+)}_{ru}\subseteq X_+-X_+$ by Lemma \ref{Lemma-2_2}~$iii)$, 
and hence $\pm y\leqslant u$ for some $u\in X_+$. So, we have $\frac{1}{n}y\convr 0(u)$ in $(X,X_+)$.
As $\frac{1}{n}y-x\in W$ for every $n\in\mathbb{N}$ then $(\frac{1}{n}y)_n\subseteq x+W$. 
It follows from $W=\overline{(X_+)}_{ru}$ that $W$ is $\tau_{ru}$-closed.
Then, $x+W$ is $\tau_{ru}$-closed because the topology $\tau_{ru}$ is translation invariant.
Since $x+W\ni\frac{1}{n}y\convr 0$ in $(X,X_+)$ and $x+W$ is $\tau_{ru}$-closed, we obtain
$0\in x+W$, and hence $x\in -W$.
It follows $[x]\in-[W]$, or $[x]\leqslant_{[W]}0$. Therefore, $(X/A,[W])$ is Archime\-dean.

\medskip
$ii)$\
Clearly, $[X_+]$ is a subcone of $[W]$. The rest follows from $i)$.

\medskip
$iii)$\
Suppose $X_+-X_+=X$. Since $X_+\subseteq W$, $W-W=X$. By Lemma \ref {Lemma-2_1}, $[W]-[W]=X/A$.

\smallskip
Suppose $[W]-[W]=X/A$. Lemma \ref {Lemma-2_1} implies $W-W=X$.
By Lemma \ref{Lemma-2_2}~$iii)$, $W\subseteq X_+-X_+$. Then, 
$$
   X_+-X_+\subseteq X=W-W\subseteq (X_+-X_+)-(X_+-X_+)\subseteq X_+-X_+,
$$
and hence $X_+-X_+=X$.
\end{proof}

\noindent
In connection with Theorem \ref{Theorem-2_1}, it would be interesting to describe those POVSs in which $[X_+]=[W]$.

\medskip
It is proved in \cite{PT-2009} that every OVS with an order unit has a unique (up to an order isomor\-phism) Archimedi\-za\-tion.
This result has being extend to arbitrary OVS in \cite{E-2017}. 

\begin{definition}\label{Definition-2_2}
{\rm (\cite[Definition 5]{E-2017})}
Let $X$ be a POVS. Consider the category $\mathcal{C}_{Arch}(X)$ whose objects are pairs $\left\langle Y,\phi\right\rangle$,
where $Y$ is an Archimedean OVS and $\phi:X\to Y$ is a positive linear operator, and morphisms
$\left\langle Y_1,\phi_1\right\rangle\to\left\langle Y_2,\phi_2\right\rangle$ are positive linear operators $T: Y_1\to Y_2$ satisfying
$T\circ \phi_1=\phi_2$. If $\left\langle Y_0,\phi_0\right\rangle$ is an initial object of $\mathcal{C}_{Arch}(X)$ 
then the OVS $Y_0$ is an {\color{blue}{Archimedeanization}} of $X$.
\end{definition}

\begin{assertion}\label{Assertion-2_3}
{\em (\cite[Theorem 1]{E-2017})}
Every OVS has an Archimedeanization.
\end{assertion}

\noindent
Unfortunately, the proof of Theorem 1 in \cite{E-2017} has a gap. We are going to discuss this issue in Section 3.
Here, we present an alternative proof of Assertion \ref{Assertion-2_3} based on Theorem \ref{Theorem-2_1}.

\begin{theorem}\label{Theorem-2_2}
Let $(X,X_+)$ be a POVS, $W=\overline{(X_+)}_{ru}$, and $A=W\cap -W$. Then $(X/A,[W])$ is an Archimedeanization of $(X,X_+)$.
\end{theorem}

\begin{proof}
The quotient map $\phi_0:X\to(X/A,[W])$ is positive. By Theorem \ref{Theorem-2_1}, 
$\left\langle(X/A,[W]),\phi_0\right\rangle$ is an object of $\mathcal{C}_{Arch}(X)$.
We need to show that $\left\langle(X/A,[W]),\phi_0\right\rangle$ is initial. To this end, take an arbitrary 
object $\left\langle Y,\phi\right\rangle$ of $\mathcal{C}_{Arch}(X)$. It remains to find a positive linear operator
$T:(X/A,[W])\to Y$ such that $T\circ\phi_0=\phi$. Since $Y$ is Archimedean,
it follows from Lemma \ref{Lemma-2_5} that $Y_+$ is ru-closed, and hence is $\tau_{ru}$-closed.
First, we show $\phi(W)\subseteq Y_+$. Let $x\in W$. There exists a net $(x_{\a})$ such that 
$X_+\ni x_{\a}\convtaur x$. By Lemma \ref{Lemma-2_4}, $\phi:X\to Y$ is $\tau_{ru}$-continuous. Therefore, we obtain 
$Y_+\ni\phi(x_{\a})\convtaur\phi(x)$. As $Y_+$ is $\tau_{ru}$-closed, $\phi(x)\in Y_+$. 
Thus, $\phi(W)\subseteq Y_+$ and $\phi(-W)\subseteq -Y_+$, and hence $\phi(A)\subseteq Y_+\bigcap-Y_+=\{0\}$. 
So, an operator $T:(X/A,[W])\to Y$ such that $T[x]:=\phi(x)$ is well defined and positive. Clearly, $T\circ\phi_0=\phi$. 
\end{proof}

\noindent
Note that $W=\overline{(X_+)}_{ru}=(X_+)^{({\w}_1)}_{ru}$ by Lemma \ref{Lemma-2_3}. Theorem \ref{Theorem-2_2} motivates the following.

\begin{definition}\label{Definition-2_3}
A POVS $X$ is of the {\color{blue}{Archimedeanization type}} ${\a}_X$ whenever
${\a}_X$ is the first ordinal ${\a}$ such that $\overline{(X_+)}_{ru}=(X_+)^{({\a})}_{ru}$.
\end{definition}

\noindent
Lemma \ref{Lemma-2_3} implies that ${\a}_X\leqslant{\w}_1$ for every POVS $X$. 
It can be seen that ${\a}_X=2$ for the vector lattice $X$ in Nakayama's example 
(cf., Example 6.4 in \cite{LM-1967}). Let us consider another class of POVSs, in which the computation of the Archimedeanization type is easy. 

\begin{definition}\label{Definition-2_4}
{\rm (\cite[Definition 3.1]{EEG-2026})}
A POVS $X$ satisfies the {\color{blue}{condition \text{\rm(R)}}} $($briefly, {\color{blue}{$X\in (R)$}}$)$ if, for each sequence $(y_k)$ in $X$, there exist $y\in X$ 
and a sequence $(\lambda_k)$ in $\mathbb{R}\setminus\{0\}$ such that $\pm\lambda_k y_k\leqslant y$ for all $k$.
\end{definition}

\noindent
An element $u$ of a POVS $X$ is an {\color{blue}{order unit}} whenever $X=\bigcup\limits_{n\in\mathbb{N}}[-nu,nu]$.
It is easy to see that every POVS possessing an order unit and every Banach lattice satisfies the condition \text{\rm(R)}.  

\begin{proposition}\label{Proposition-2_2}
Let $X$ be a POVS satisfying the condition \text{\rm(R)}.
Then $X_+$ is generating and $\overline{S}_{ru}=S^{(1)}_{ru}$ for every $S\subseteq X$.
\end{proposition}

\begin{proof}
Let $x\in X$. By Definition \ref{Definition-2_4},
there exist $y\in X$ and $0\ne\lambda\in\mathbb{R}$ with $\pm\lambda x\leqslant y$. It follows $y,y+\lambda x\in X_+$.
Then, $\lambda x=(y+\lambda x)-y\in X_+-X_+$, and hence $x\in X_+-X_+$. Therefore, $X_+$ is generating. 

\medskip
Let $S\subseteq X$. By Lemma \ref{Lemma-2_2}, in order to prove $\overline{S}_{ru}=S^{(1)}_{ru}$ it suffices to show that $S^{(1)}_{ru}$ is ru-closed. 
Let $S^{(1)}_{ru}\ni z_n\convr z$. For each $n$, find a a sequence $(z_n^k)_k$ in $X_+$ such that $z_n^k\convr z_n$. 
By \cite[Lemma 3.2]{EEG-2026},
any countable set of ru-convergent sequences in $X$ has a common regulator. Thus, we may assume $z_n\convr z(u)$ and $z_n^k\convr z_n(u)$
for some $u\in X_+$ and all $k\in\mathbb{N}$. By passing to subsequences, if necessary, we may assume 
$$
   \pm n(z_n-z)\leqslant u \ \ \ (n\in\mathbb{N}) \ \ \ \text{\rm and} \ \ \ 
   \pm k(z_n^k-z_n)\leqslant u \ \ \ (k,n\in\mathbb{N}). 
$$
In particular, we have $\pm n(z_n-z)\leqslant u$ and $\pm n(z_n^n-z_n)\leqslant u$ for all $n\in\mathbb{N}$.
Summing the inequality up gives $\pm n(z_n^n-z)\leqslant 2u$ for all $n\in\mathbb{N}$. Then, $X_+\ni z_n^n\convr z$. It follows $z\in S^{(1)}_{ru}$,
and hence $S^{(1)}_{ru}$ is ru-closed.
\end{proof}

\begin{corollary}\label{Corollary-2_1}
If $X\in (R)$ then ${\a}_X\leqslant 1$, and if additionally $X$ is non-Archimedean, then ${\a}_X=1$.
In particular, ${\a}_X=1$ for every finite-dimensional non-Archimedean POVS $X$. 
\end{corollary}

\noindent
In general, it does not follow from ${\a}_X\leqslant 1$ that $X\in (R)$ even when $X_+$ is generating. Indeed, ${\a}_{c_{00}}=0$ yet $c_{00}\notin (R)$.

\section{On an almost Archimedeanization of a POVS}

\hspace{4mm}
In this section, we discuss a construction of almost Archimedeanization motivated by \cite{E-2017}. Let $X$ be a POVS. Following \cite{E-2017}, we denote 
$$
   D_X:=\{x\in X|(\exists\xi\in X_+)(\forall n\in\mathbb{N})\ nx+\xi\in X_+\}, \ \ \text{\rm and}
$$
$$
   N_X:=\bigl\{x\in X: (\exists\xi\in X_+)(\forall n\in\mathbb{N})\pm nx+\xi\in X_+\bigl\}\ \ \ -
$$
the set of {\color{blue}{infinitesimals}} of $(X,X_+)$. It is straightforward to see: $D_X$ is a wedge; $X_+\subseteq D_X$; 
$N_X$ is an order ideal in $(X,X_+)$; $(X,X_+)$ is almost Archimedean if and only if $N_X=\{0\}$;  
$$
   D_X=\{x\in X|(\exists\xi\in X_+)(\forall r>0)\ x+r\xi\in X_+\};
$$
and
$$
   N_X=\{x\in X|(\exists\xi\in X_+)(\forall r\in\mathbb{R})\ rx+\xi\in X_+\}.
   \eqno(1)
$$

\begin{proposition}\label{Proposition-3_1}
Let $X$ be a POVS. The following assertions hold.
\begin{enumerate}[$i)$]
\item\
$X$ is Archimedean if and only if $D_X=X_+$.
\item\
$D_X=(X_+)^{(1)}_{ru}$
\item\
$D_X\cap-D_X=N_X$.
\item\
$X$ is almost Archimedean $\iff$ $D_X$ is a cone $\iff$ $X_+$ does not contain a straight line.
\end{enumerate}
\end{proposition}

\begin{proof}
$i)$\
$(\Longrightarrow)$:\ 
Let $X$ be Archimedean. As $X_+\subseteq D_X$, we need to prove $D_X\subseteq X_+$.
So, let $x\in D_X$. Then, for some $\xi\in X_+$, $\xi\geqslant n(-x)$ for every $n\in\mathbb{N}$.
Since $X$ is Archimedean, we infer $-x\leqslant 0$, or $x\in X_+$.

\smallskip
$(\Longleftarrow)$:\ 
Let $D_X=X_+$. Assume that $ny\leqslant\xi$ for some $\xi\in X_+$ and all $n\in\mathbb{N}$.
Then, $n(-y)+\xi\geqslant 0$ for all $n\in\mathbb{N}$, and hence $-y\in D_X$. Since $D_X=X_+$,
we infer $y\leqslant 0$. Thus, $X$ is Archimedean.

\medskip
$ii)$\
If $x\in D_X$ then, for some $\xi\in X_+$ we have $x+n^{-1}\xi\in X_+$ for all $n\in\mathbb{N}$. 
Since $x+n^{-1}\xi\convr x$, then $x\in(X_+)^{(1)}_{ru}$. Thus, $D_X\subseteq(X_+)^{(1)}_{ru}$. 

\smallskip
If $x\in (X_+)^{(1)}_{ru}$, there exist $\xi\in X_+$ and a sequence $(x_n)$ in $X_+$ satisfying $\pm n(x_n-x)\leqslant\xi$ for all $n\in\mathbb{N}$. 
Then $0\leqslant nx_n\leqslant nx+\xi$, and hence $nx+\xi\in X_+$ for all $n\in\mathbb{N}$. So, $x\in D_X$.
Thus, $(X_+)^{(1)}_{ru}\subseteq D_X$.

\medskip
$iii)$ is trivial.

\medskip
$iv)$\
Let $X$ be almost Archimedean. By $iii)$, $D_X\cap-D_X=N_X=\{0\}$, and hence $D_X$ is a cone.
Let $D_X$ be a cone. Then $X_+$ is a cone. Suppose $X_+$ contains a straight line, say $u+\mathbb{R}\cdot v\subseteq X_+$ for some $u,v\in X$, $v\ne 0$.
Since $X_+$ is a cone then $u\ne 0$. As $u-nv\subseteq X_+$ for all $n\in\mathbb{Z}$ then $\pm nv\leqslant u$
for all $n\in\mathbb{N}$, or $v\in N_X$. By $iii)$, $N_X=D_X\cap-D_X=\{0\}$, and hence $v=0$, a contradiction. So, $X_+$ does not contain a straight line.
Suppose $X$ is not almost Archimedean. Then, $N_X\ne\{0\}$, and hence $X_+$ contains a straight line by formula (1).
\end{proof}

\noindent
In general, $D_X$ is not necessarily a cone even if $\text{\rm dim}(X)=2$ (take, for example, $\mathbb{R}^2_{lex}$). 
We collect the further properties of subsets $D_X$ and $N_X$ of a POVS $X$ in the following proposition.

\begin{proposition}\label{Proposition-3_2}
{\rm (\cite[Propositions 3.1 and 3.2]{E-2017})}
Let $X$ be a POVS. Then following assertions hold.
\begin{enumerate}[$i)$]
\item\
$N_X$ is an order ideal in the POVS $(X,D_X)$ satisfying $N_X\subseteq\overline{\{0\}}_{ru}$.
\item\
If $X$ has an order unit and $A$ is an order ideal in $X$ then $N_{X/A}=\overline{\{[0]\}}_{ru}$ in $(X/A,[X_+])$.
\item\
If $A$ is an order ideal in $X$ such that $(X/A,[X_+])$ is almost Archimedean then $A$ is ru-closed in $X$.
\item\
$X$ is almost Archimedean  $\iff$ $\overline{\{0\}}_{ru}=\{0\}$.
\item\
If $X$ has an order unit then $(X/N_X,[D_X])$ is Archimedean.
\item\
If $(X/N_X,[D_X])$ is Archi\-me\-dean then $(X/N_X,[X_+])$ is almost Archi\-me\-dean. In particular,
if $X$ has an order unit then $(X/N_X,[X_+])$ is almost Archimedean.
\item\
If $X$ is a vector lattice then $(X/N_X,[D_X])$ is a vector lattice.
\end{enumerate}
\end{proposition}

The next elementary assertion strengthens Proposition \ref{Proposition-3_2}~$i)$.

\begin{assertion}\label{Assertion-3_1}
A if $X$ is a POVS then $N_X=\{0\}^{(1)}_{ru}$.
\end{assertion}

\begin{proof}
Let $x\in X$. Then, $x\in N_X$ iff $\pm nx+\xi\in X_+$ for some $\xi\in X_+$ and all $n\in\mathbb{N}$ iff
$\pm n(0-x)\leqslant\xi$ for some $\xi\in X_+$ and all $n\in\mathbb{N}$ iff $x\in\{0\}^{(1)}_{ru}$.
\end{proof}

\medskip
The following definition is an adaptation of Definition \ref{Definition-2_2} (see, also \cite[Definition 5]{E-2017}) to the almost Archimedean setting.

\begin{definition}\label{Definition-3_1}
Let $X$ be a POVS and $\mathcal{C}^a_{Arch}(X)$ a category, whose objects are pairs $\left\langle Y,\phi\right\rangle$,
where $\phi:X\to Y$ is a positive linear operator to an almost Archimedean OVS $Y$, and morphisms
$\left\langle Y_1,\phi_1\right\rangle\to\left\langle Y_2,\phi_2\right\rangle$ are positive linear operators $\psi: Y_1\to Y_2$ satisfying
$\psi\circ\phi_1=\phi_2$. If $\left\langle Y_0,\phi_0\right\rangle$ is an initial object of $\mathcal{C}^a_{Arch}(X)$ 
then the OVS $Y_0$ is said to be {\color{blue}{almost Archimedeanization}} of $X$.
\end{definition}

One of the key ingredients in the proof of \cite[Theorem 1.1]{E-2017} is \cite[Proposition 3.4]{E-2017}
stating that, for every almost Archimedean OVS $X$ the OVS $(X,D_X)$ is Archimedean.
The proof of \cite[Proposition 3.4]{E-2017} has a gap and the existence of Archimedeanization of a POVS
is not proved in \cite{E-2017}. However, an inspection of the proof of \cite[Theorem 1.1]{E-2017} tells us that 
its modification provides a construction of almost Archimedeanization. This will be done in Theorem \ref{Theorem-3_1} below.

We need one more notion. Consider the following sets in a POVS $X$:\ $N_0(X)=\{0\}$ and, for every ordinal $\lambda>0$,  
$N_{\lambda}(X)=\bigl\{x\in X: [x]_{\bigcup_{{\g}<\lambda}N_{{\g}}(X)}\in N\bigl(X/_{\bigcup_{{\g}<\lambda}N_{{\g}}(X)}\bigl)\bigl\}$ 
the set of {\color{blue}{infinitesimals of $X$ of order $\lambda$}}, 
where $\lambda$ is an arbitrary ordinal $\geqslant 1$ (here, we use the straightforward fact that each $N_{\lambda}(X)$ as well as $\bigcup_{{\g}<\lambda}N_{{\g}}(X)$
is an order ideal in $X$). Since $N_{\lambda_1}(X)\subseteq N_{\lambda_2}(X)$ for $\lambda_1\leqslant\lambda_2$ then 
$N_{\lambda+1}(X)=N_{\lambda}(X)$ for $\lambda\geqslant\card(X)$. Therefore, there exists
the first ordinal, denoted by $\lambda_X$, such that $N_{\lambda_X+1}(X)=N_{\lambda_X}(X)$.
We call $\lambda_X$ the {\color{blue}{almost Archimedeanization type}} of $X$ (cf., \cite{E-2017,E-2023}).

\begin{lemma}\label{Lemma-6}
Let $T$ be a positive linear operator from a POVS $X$ to a POVS $Y$. Then, $T(N_{\lambda}(X))\subseteq N_{\lambda}(Y)$ for every ordinal ${\g}$.
\end{lemma}

\begin{proof}
Trivially, $T(N_{0}(X))\subseteq N_{0}(Y)$.
Let $x\in N_{{\a}+1}(X)$. Then $\pm n[x]_{N_{{\a}}(X)}\leqslant [u]_{N_{{\a}}(X)}$ for some $u\in X_+$ and all $n\in\mathbb{N}$, 
and hence $\pm n[Tx]_{N_{{\a}}(Y)}\leqslant [Tu]_{N_{{\a}}(Y)}$ for $n\in\mathbb{N}$. So, $Tx\in N_{{\a}+1}(Y)$. 
If ${\g}$ is a limit ordinal and $x\in N_{{\g}}(X)$ then $x\in N_{{\a}}(X)$ for some ${\a}<{\g}$. 
Thus, $\pm n[x]_{N_{{\a}}(X)}\leqslant [u]_{N_{{\a}}(X)}$ for some $u\in X_+$ and all $n\in\mathbb{N}$, and hence
$\pm n[Tx]_{N_{{\a}}(Y)}\leqslant [Tu]_{N_{{\a}}(Y)}$ for all $n\in\mathbb{N}$. It follows $Tx\in N_{{\a}+1}(Y)\subseteq N_{{\g}}(Y)$.
\end{proof}

\begin{theorem}\label{Theorem-3_1}
Every POVS possesses a unique, up to an order isomorphism, almost Archime\-deanization.
\end{theorem}
 
\begin{proof}
Let $X$ be a POVS. By the definition of $\lambda_X$, $(X/N_{\lambda_X}(X),[X_+]_{N_{\lambda_X}(X)})$ is an almost Archimedean OVS. 
Denote by $p_X$ the quotient map $X\to X/N_{\lambda_X}(X)$. Since $N_{\lambda_X}(X)$ is an order ideal in $X$, $p_X$ is positive.
Let $(Y,Y_+)$ be an almost Archimedean OVS and $\phi:X\to Y$ a positive linear operator. By Lemma \ref{Lemma-6}, 
$\phi(N_{\lambda_X}(X))\subseteq N_{\lambda_X}(Y)$. Since $Y$ is almost Archimedean, $N_{\lambda_X}(Y)=\{0\}$.
So, a mapping $\tilde{\phi}:X/N_{\lambda_X}(X)\to Y$ 
such that $\tilde{\phi}([x]_{N_{\lambda_X}})=\phi(x)$ is a well defined positive linear operator satisfying $\tilde{\phi}\circ p_X=\phi$. 
Thus, $\left\langle(X/N_{\lambda_X}(X),[X_+]_{N_{\lambda_X}(X)}),p_X\right\rangle$ is an initial object of $\mathcal{C}^a_{Arch}(X)$,
and hence the OVS $(X/N_{\lambda_X}(X),[X_+]_{N_{\lambda_X}(X)})$ is an almost Archimedization of $X$.
\end{proof}

It is mistakenly claimed in \cite[Proposition 3.4]{E-2017} that $(X,D_X)=(X,(X_+)^{(1)}_{ru})$ is an Archimede\-an POVS,
whenever $X$ is an almost Archimedean OVS. Under the assumption that $X$ almost Archimedean, 
we can only say $\big(X,\overline{(X_+)}_{ru}\big)$ is Archimedean.
More precisely, we have the following corollary of Theorem \ref{Theorem-2_1}.

\begin{corollary}\label{Corollary-3_1}
If $X$ is a POVS then $\big(X,\overline{(X_+)}_{ru}\big)$ is an Archimedean POVS.
\end{corollary}

\begin{proof}
Let $X$ be a POVS and $nx\leqslant_{\overline{(X_+)}_{ru}}y$ for some $x\in X$, $y\in \overline{(X_+)}_{ru}$, and all $n\in\mathbb{N}$. Then,
$n[x]\leqslant_{\big[\overline{(X_+)}_{ru}\big]}[y]$ for all $n\in\mathbb{N}$ 
in the POVS $\big(X/\big(\overline{(X_+)}_{ru}\cap-\overline{(X_+)}_{ru}\big),\big[\overline{(X_+)}_{ru}\big]\big)$. 
By Theorem \ref{Theorem-2_1}, $[x]\in-\big[\overline{(X_+)}_{ru}\big]$, and hence 
$$
   x+\overline{(X_+)}_{ru}\cap-\overline{(X_+)}_{ru}\subseteq-\overline{(X_+)}_{ru}+\big(\overline{(X_+)}_{ru}\cap-\overline{(X_+)}_{ru}\big)=-\overline{(X_+)}_{ru}. 
$$
It follows $x\in -\overline{(X_+)}_{ru}$, as desired.
\end{proof}

\medskip
It can be seen easily that: $0\leqslant\lambda_X\leqslant{\w}_1$ for every POVS $X$; $\lambda_X=0$ if and only if $X$ is almost Archimedean;
$N_1\bigl(\ell^\infty/c_{00}\bigl)=N_1\bigl(c_0/c_{00}\bigl)=c_0/c_{00}$, and hence $N_2\bigl(\ell^\infty/c_{00}\bigl)=N_2\bigl(c_0/c_{00}\bigl)=\{0\}$
and $\lambda\bigl(\ell^\infty/c_{00}\bigl)=\lambda\bigl(c_0/c_{00}\bigl)=2$; and $\lambda_X=3$ for the vector lattice $X$ in Nakayama's example.

\section{Examples of OVSs and Vector Lattices whose almost Archimedeanization type is $\a$ for $0\leqslant\a\leqslant\w_1$.}

\hspace{4mm}
In this section, we present examples showing the variety of OVSs and vector lattices of finite and countable almost Archimedeanization types.
Recall that by $\lambda_V$ we denote the almost Archimedeanization type of a POVS $V$. 
In \cite{E-2017} the question was asked whether or not for any ordinal $\a$ there exists an OVS $V$ for which $\lambda_V=\a$?
It is easy to see that $\lambda_V\leqslant\w_1$ for each OVS $V$, since being infinitely small is an inherently countable condition.
On the other hand, we are going to show, as an explicit answer to the question, that it is possible to construct an OVS, and even a lattice, 
$ V $ such that $\lambda_V=\a$ for any $0\leqslant\a\leqslant\w_1$. We begin with an example of an OVS, since it is slightly easier to observe, 
while having most of the important features necessary for the construction.

\medskip
Let $\mathcal{P}$ be the vector space of all {\color{blue}{real polynomials}}. Consider an ordering $\leqslant_0$ on $\mathcal{P}$ given by
$$
   0 \leqslant_0 p \iff ( \forall n \in \mathbb{N} )( p(n) \geqslant 0 ).
$$
One can immediately check that there are no non-zero infinitely small elements and thus $\lambda_{(\mathcal{P}, \leqslant_0)}=0$.
Now let an ordering $\leqslant_1$ on $\mathcal{P}$ be given by
$$
   0 \leqslant_1 p \iff \text{leading coefficient of } p \text{ is non-negative}.
$$
In this case every element becomes infinitely small and hence $\lambda_{(\mathcal{P}, \leqslant_1)}=1$.

\begin{example}\label{Example-4_1}
For any countable limit ordinal $ \a $ fix a strictly increasing cofinal sequence $ A_\a $, i.e. a mapping $ A_\a : \mathbb{N} \to \a $ such that
$$
   A_\a(n)<A_\a(m) < \a \text{ for any }  n < m \text{ and } \lim_{n \to \infty} A_\a(n) = \a. 
$$
Such sequences can easily be constructed from any bijection between $\a$ and $\mathbb{N}$.

\smallskip
Consider vector spaces
$$
   V_\a = \bigoplus_{1 \leqslant \b \leqslant \a} \mathcal{P} \ \ \ \ (1 \leqslant \a < \w_1).
$$
Given an element $ p \in V_\a $ let $ p_\b $ denote its projection onto the $ \b $-th direct summand for any $ 1 \leqslant \b \leqslant \a $.
Turn each $ V_\a $ for $ \a < \w_1 $ into an OVS $(V_\a, \leqslant_\a)$ by an ordering given by
$$
	0 \leqslant_\a p \iff
	( \forall \b \leqslant \a )
	( 0 \leqslant_1 p_\b \text{ and } ( \forall n \geqslant n_\b(p) )( p_\b(n) \geqslant 0 ) ),
$$
where
$$
	n_\b(p) =
	\begin{cases}
		\infty  & \text{ if } \b = 1,\\
		\deg p_\g, & \text{ if } \b = \g + 1 > 1,\\
		\min\{n : p_{A_\b(n)} = 0\}, & \text{ if } \b \text{ is a limit ordinal}.
	\end{cases}
$$
WLOG, we assume that $(V_\a, \leqslant_\a)$ is a subspace of $(V_\b, \leqslant_\b) $ for any $ \a < \b $.
\end{example}

\begin{theorem}\label{Theorem-4_1}
$\lambda_{V_\a} = \a $ for any $ 1 \leqslant \alpha < \w_1 $.
\end{theorem}

\begin{proof}
Explicitly writing $ N_\b (V_\a) $,  we get $N_\b(V_\a) = V_\b$ for any $1 \leqslant \b \leqslant \a $.
It is obviously true for $ \b = 1 $, since the first component is precisely $ (\mathcal{P}, \leqslant_1) $, and if $ p $ 
has non-zero coordinate $ p_\g $ for some $ \g > 1 $, then for any $ q \in V_\a $ and any $ n \in \mathbb{N} $ 
there is a large enough number $ m \in \mathbb{N} $ such that either $ (q_\g-mp_\g)(n) < 0 $ or $ (-q_\g-mp_\g)(n) > 0 $, 
so either $ mp \not<_\a q $ or $ -q \not<_\a mp $.

\medskip
If it is shown for every $ \g < \b $ then relations on components $ p_\g $ become irrelevant and every element $ p $, 
such that $ p_\g = 0 $ for all $ \g > \b $, lies in $ N_\b(V_\a) $, as the order on the $ \b $-th component collapses 
into just $ \leqslant_1 $ since we can exclude arbitrarily many values $ p_\beta(n) $ from the consideration 
by tampering with components on indices $ \g < \b $. On the other hand, again, if $ p_\g \neq 0 $ for some $ \g > \b $ 
then for any $ q \in V $ there is a bound $ k \in \mathbb{N} $ such that we can't exclude values in $ n > k $ from the consideration, 
but then again there is a large enough number $ m \in \mathbb{N} $ such that either $ (q_\g-mp_\g)(n) < 0 $ or $ (-q_\g-mp_\g)(n) > 0 $, 
which shows that $ p \notin N_\b(V_\a) $.
From this explicit inductive form of $ N_\b(V_\a) $ it immediately follows $\lambda_{V_\a} = \a $ for any $ \a < \w_1 $.
\end{proof}

\noindent
Now take a colimit $V_{\w_1}=\bigoplus_{1 \leqslant \b < \w_1} \mathcal{P}$ of $V_\a $, with the ordering $ \leqslant_{\w_1}$ from the colimit, i.e.
$$
   0 \leqslant_{\w_1} p \iff ( \forall \b < \w_1 ) ( 0 \leqslant_1 p_\b \text{ and } ( \forall n \geqslant n_\b(p) )( p_\b(n) \geqslant 0 ) ).
$$

\begin{corollary}\label{Corollary-4_1}
$\lambda_{V_{\w_1}} = \w_1 $.
\end{corollary}

\bigskip
Next, we present examples in the vector lattice setting. For any ordinal $ \a > 1 $ fix a function $ \a : \mathbb{N} \to \a $ such that
$$
	\a(n) =
	\begin{cases}
		\b, & \text{ if } \a = \b + 1 > 1,\\
		A_\a(n), & \text{ if } \a \text{ is a limit ordinal}.
	\end{cases}
$$
where $ A_\a $ is an increasing sequence from Example \ref{Example-4_1}.

\medskip
Consider the following recursive family of spaces $ S_\a $ and their fixed elements $ 1_\a $:
$$
\begin{gathered}
	S_1 = \mathbb R \quad\text{ and }\quad 1_1 = 1,
	\\
	S_\a = \mathcal{P}1_\alpha + \bigoplus_{n \in \mathbb N} S_{\a(n)} \subset \prod_{n \in \mathbb N} S_{\a(n)}
	\quad\text{ and }\quad
	1_\a = \prod_{n \in \mathbb N} 1_{\a(n)},
\end{gathered}
$$
whose elements we will regard as sequences $ x_n $, with the recursive pointwise order
$$
   0 \leqslant^\a x \iff (\forall n \in \mathbb N)(0 \leqslant^{\a(n)} x_n).
$$
So any element $ x \in S_\a $ is a sequence of elements $ x_n \in S_{\a(n)} $ such that for some (unique) 
$ p \in \mathcal{P}$ we have $ x_n = p(n)1_{\a(n)} $ for all but finitely many $ n \in \mathbb N $. If necessary, we denote such $ p $ by $ p(x) $.

\noindent
Note that by definition each $ S_\a $ is a subspace of $ \mathbb R^{I_\a} $ for some $ I_\a $, 
moreover, from definition $ I_\a = \coprod I_{\a(n)} $, so $ I_n \simeq \mathbb N^{n-1} $, 
and $ I_\a \simeq \mathbb N^{<\w} $ for all $ \w \leqslant \a < \w_1 $. And also note, that $ 0 \leqslant^\a x \iff (\forall i \in I_\a)(0 \leqslant x(i)) $.

\medskip
Consider a family of subspaces $ H_\a^\b \subseteq S_\a $ such that
$$
	H_\a^\b =
	\begin{cases}
		\bigoplus\limits_{n \in \mathbb N} H_{\a(n)}^\b, & \text{ if } \b < \a, \\
		S_\a, & \text{ if } \b \geqslant \a.
	\end{cases}
$$
Note that $ H_\a^1 $ is precisely $c_{00}(I_\a)$, and that $ H_\a^\g \subset H_\a^\b $ for all $ \g < \b \leqslant \a $.
Let's also, for simplicity, denote $ H_\a^{<\b} = \bigcup\limits_{\g < \b} H_\a^\g $.

\bigskip
Consider spaces $V_\a = H_\a^1 \times H_\a^\a$ with the pointwise lexicographic order:
$$
	0 \leqslant (x, y) \iff (\forall i \in I_\a)(0 < x(i) \text{ or } (0 = x(i) \text{ and } 0 \leqslant y(i))).
$$
Note that on $ H_\a^1 \times 0 $ and $ 0 \times H_\a^\a $ this is exactly the order we've already had from $ S_\a $.

\begin{proposition}\label{Proposition-4_1}
Each $ V_\a $ is a vector lattice.
\end{proposition}

\begin{proof}
First, note that $ S_\a $ is a vector lattice. Naturally, $ (x \lor 0)_n = x_n \lor 0 $, so, for $ x_n = p(n)1_{\a(n)} $, 
if $ 0 \leqslant_1 p $ then $ (x \lor 0)_n = p(n)1_{\a(n)} $ for all but finitely many $ n $, 
where we can choose it recursively, and if $ 0 >_1 p $ then $ (x \lor 0)_n = 0 $ for all but finitely many $ n $.
\\
Now, $ 0 \lor (x, y) = (0 \lor x, 0 \lor y + \delta) $, where
$$
\delta(i) =
\begin{cases}
	y(i) - (0 \lor y)(i), & \text{ if } 0 < x(i), \\
	0, & \text{ if } x(i) \leqslant 0
\end{cases}
$$
which is an element from $ H_\a^1 $, so $ \delta_n = 0 $ for all but finitely many $ n $.
\end{proof}

\begin{lemma}\label{there_is_gap}
If $ \b < \a $, then $| 1_\a - x | \not\leqslant t 1_\a$ for any $ x \in H_\a^\b $ and any real number $ 0 \leqslant t < 1 $.
\end{lemma}

\begin{proof}
We have $ x \in \bigoplus H_{\a(n)}^\b $ since $ \b < \a $, and thus $ x_n = 0 $ for all but finitely many $ n $. 
Then if $ | 1_\a - x | \leqslant t 1_\a $ for sufficiently large $ n $ we have $ 1_{\a(n)} \leqslant t 1_{\a(n)} $ which means $ 1 \leqslant t $ at every $ i \in I_\a $.
\end{proof}

\begin{lemma}\label{ru_is_inf}
If $J$  is an order ideal of a vector lattice $V$, then $[x] \in N_1(V/J) 	\iff  x \in {J}_{ru}^{(1)}$.
\end{lemma}

\begin{proof}
It follows immediately from Proposition \ref{Proposition-3_1}.
\end{proof}

\begin{lemma}\label{H_is_closed}
$ (H_\a^{<\b})_{ru}^{(1)} = H_\a^\b $ for every $ 1 < \b < \w_1 $ and $ 1 \leqslant \a < \w_1 $.
\end{lemma}

\begin{proof}
We prove this by induction on $ \alpha $.

\medskip\noindent
\underline{Case $ \a < \b $:}\\
It is obviously true since there is $ \a \leqslant \g < \b $, and thus $ H_\a^\g = H_\a^\b = S_\a $.

\medskip\noindent
\underline{Case $ \a = \b $:}\\
Clearly, $ H_\a^\b = S_\a $, so $ \subseteq $ is trivial.
On the other hand, $ \bigoplus S_{\a(n)} \subseteq H_\a^{<\a} $. Indeed, if we take $ x \in S_\a $ such that $ x_n = 0 $ for all $ n > m $, then, after noticing that
$$
	H_{\a(n)}^{\a(m)} = S_{\a(n)} \quad\text{ for every } n \leqslant m,
$$
we may conclude that $ x \in \bigoplus H_{\a(n)}^{\a(m)} = H_\a^{\a(m)} $.
Now, if $ x \in S_\a $ and $ x_{(m)} $ denotes its initial segment of length $ m $, then $ x_{(m)}\overset{ru}{\rightarrow} x (w) $ 
where the witness $ w $ is given by any positive polynomial growing faster than $ p(x) $, for example, as $ w = (1+t^{\deg p(x) + 1})1_\a $.

\medskip\noindent
\underline{Case $ \a > \b $:}\\
Let's prove $ \supseteq $ first.
For any element $ x \in H_\a^\b $, we know that $ p(x) = 0 $ and $ x_n \in H_{\a(n)}^\b $ for all $ n $, so let $ m $ be any number such that $ x_n = 0 $ for all $ n > m $.
For every $ x_n $ we have sequences of elements $ y_{n, k} \in H_{\a(n)}^{<\b} $ and witnesses 
$ w_n \in S_{\a(n)} $ such that $ y_{n, k}  \overset{ru}{\rightarrow} x_n (w_n) $.
Choose $ \g_k < \b $ such that $ y_{n, k} \in H_{\a(n)}^{\g_k} $ for all $ n < m $, and construct $ y_k \in H_\a^{\g_k} $
$$
	(y_k)_n = 
	\begin{cases}
		y_{n, k}, & \text{ if } n < m, \\
		0, & \text{ if } n \geqslant m,
	\end{cases}
$$
as well as $ w \in S_\a $ from $ w_n $, similarly.
It is easy to see that $ y_k  \overset{ru}{\rightarrow} x (w) $.
\\
Now in the $ \subseteq $ direction.
Let's take an arbitrary sequence $ y_k \in H_\a^{<\b} $ and $ x, w \in S_\a $ such that $ y_k  \overset{ru}{\rightarrow} x (w) $.
Suppose that $ p(x) = p \neq 0 $. Then, for any sufficiently large $ n $, we have $ p(n) \neq 0 $, $ \a(n) \geqslant \b $, $ x_n = p(n) 1_{\a(n)} $, 
and $ w_n = q(n) 1_{\a(n)} $. The ru-convergence gives
$$
	k \ |p(n) 1_{\a(n)} - (y_k)_n | \leqslant q(n) 1_{\a(n)},
$$
and thus
$$
	| 1_{\a(n)} - \frac{1}{p(n)}(y_k)_n | \leqslant \frac{q(n)}{k|p(n)|} 1_{\a(n)}.
$$
But $ y_k \in H_{\a}^\g $ for some $ \g < \b $, which means $ (y_k)_n \in H_{\a(n)}^\g $, which, for large enough $ k $, such that $  \frac{q(n)}{k|p(n)|} < 1 $ 
contradicts Lemma \ref{there_is_gap}.
Hence, $ p(x) = 0 $. On the other hand, $ y_{n, k}  \overset{ru}{\rightarrow} x_n (w_n) $, so, by induction, $ x_n \in H_{\a(n)}^\b $, 
which means that $ x \in \bigoplus H_{\a(n)}^\b = H_\a^\b $.
\end{proof}

\begin{theorem}\label{Theorem-4_2}
$\lambda_{V_\a} = \a $ for any $ 1 \leqslant \alpha < \w_1 $.
\end{theorem}

\begin{proof}
Explicitly writing $ N_\b (V_\a) $ for any $ 1 \leqslant \b \leqslant \a $ we get
$$
	N_\b(V_\a) = 0 \times H_\a^\b.
$$
Indeed, $ 0 \times H_\a^1 \subseteq N_1(V_\a) $, 
because of the lexicographic order, and $ N_1(V_\a) \subseteq 0 \times c_{00}(I_\a)$, 
since $ S_\a $ is Archimedean and only a finite number of coordinates of the second component may be omitted from the comparison via the first component.

\medskip
Clearly for any $ H_\a^1 \subseteq J \subseteq H_\a^\a $ we have $ (0 \times J)_{ru}^{(1)} = 0 \times J_{ru}^{(1)} $. Hence, by induction, from Lemmas \ref{ru_is_inf} and \ref{H_is_closed} we have the desired formula for $ N_\b(V_\a) $.

\medskip
Now $ V_\a / N_\a(V_\a) $ is just $ H_\a^1 \times 0 $, which is Archimedean.
\end{proof}

\noindent
Again, take any colimit of $ V_\a $ with the colimit ordering, for example $V_{\w_1} = \bigoplus_{1 \leqslant \a < \w_1} V_\a$.

\begin{corollary}\label{Corollary-4_2}
$\lambda_{V_{\w_1}} = \w_1 $.
\end{corollary}

\section{The case of pre-ordered Abelian groups}

\hspace{4mm}
In this section, we extend some of results obtained in previous sections to the setting of pre-ordered Abelian groups.
First, we adopt necessary definitions to pre-ordered groups. In what follows we use additive notations denoting by "+"
the group operation and by 0 the unit of the group.

A subset $W$ of a group $X$ is a {\color{blue}{wedge}} whenever $0\in W$ and $W+W\subseteq W$. 
A wedge $C$ is called a {\color{blue}{cone}}, if $C\cap-C=\{0\}$. 
An {\color{blue}{pre-ordered group}} (briefly, {\color{blue}{POG}}) is a group together with a wedge. 
It is convenient to denote the wedge $W$ in a POG (X,W) by $X_+$.
When $X_+$ is a cone, $X$ is called an {\color{blue}{ordered group}} (briefly, {\color{blue}{OG}}).

Every POG $X$ is equipped with a {\color{blue}{partial pre-order}} $x\leqslant y \iff y-x\in X_+$. 
Each pair of vectors $\{a,b\}$ in a POG $X$ produces an {\color{blue}{order interval}}:
$$
   [a,b]:=\{x\in X: a\leqslant x\leqslant b\}=(a+X_+)\cap(b-X_+).
$$
An element $u$ of a POG $X$ is an {\color{blue}{order unit}}, whenever $X=\bigcup\limits_{n\in\mathbb{N}}[-nu,nu]$.
A subset $I$ of a POG $X$ is an {\color{blue}{order ideal}} in $X$, whenever $I$ is a subgroup of $X$ such that $[a,b]\subseteq I$ for every $a,b\in I$.
A group homomorphism $T: X\to Y$ between POGs $X$ and $Y$ is {\color{blue}{positive}} ({\color{blue}{order bounded}}) if $T$ takes $X_+$ into $Y_+$
($T$ takes order intervals of $X$ into order intervals of $Y$).

\begin{definition}\label{Definition-5_1}
A POG $X$ is:
\begin{enumerate}[-]
\item\
{\color{blue}{almost Archimedean}} if $y=0$, whenever $\pm ny\leqslant u\in X_+$ for all $n\in\mathbb{N}$.
\item\
{\color{blue}{Archimedean}} if $y\leqslant 0$, whenever $ny\leqslant u\in X_+$ for all $n\in\mathbb{N}$.
\end{enumerate}
\end{definition}

\noindent
It follows from Definition \ref{Definition-5_1}: every almost Archimedean wedge is a cone, every subcone of an almost Archimedean cone is almost Archimedean, and
every Archimedean cone is almost Archime\-dean. The wedge $X_+=X$ is Archimedean yet not almost Archimedean unless $X=\{0\}$.

\begin{definition}\label{Definition-5_2}
A net $(x_{\a})$ in a POG $X$ {\color{blue}{ru-converges}} to $x\in X$ {\em (}shortly, $x_{\a}\convr x${\em )} 
if, for some $u\in X_+$, there exists a sequence $({\a}_n)$ of indices with $\pm n(x_{\a}-x)\leqslant u$ for ${\a}\geqslant{\a}_n$.
\end{definition}

\noindent
It is worth noting that (in compare with Definition \ref{Definition-1_8}) one can only conclude that $2u\in X_+$ from $\pm n(x_{\a}-x)\leqslant u$,
and therefore we suppose $u\in X_+$ in Definition \ref{Definition-5_2}. Whenever we specify a regulator $u$ of the ru-convergence, we write $x_{\a}\convr x(u)$. 

\begin{lemma}\label{Lemma-5_1}
Let $X$ be an POG. The following assertions hold.
\begin{enumerate}[$i)$]
\item\
If every ru-convergent net in $X$ has a unique ru-limit then $X$ is almost Archime\-dean.
\item\
If $X$ is Abelian and almost Archime\-dean then every ru-convergent net in $X$ has a unique ru-limit.
\end{enumerate}
\end{lemma}

\begin{proof}
$i)$\ 
Otherwise, $\pm ny\leqslant u$ for some $y\ne 0$, $u\in X_+$, and all $n\in\mathbb{N}$.
It follows that $\pm n(y-0)\leqslant u$ and $\pm n(y-y)\leqslant u$ for all $n\in\mathbb{N}$. Thus, $y\convr 0(u)$ and $y\convr y(u)$, and hence 
every constant net with all terms equal to $y$ has two different ru-limits, a contradiction.

\medskip
$ii)$\ 
Let $X$ be Abelian. In view of $i)$, we have to show that every ru-convergent net in $X$ has a unique ru-limit
whenever $X$ is almost Archime\-dean. On a way to contradiction, suppose $X$ is almost Archime\-dean, $x_{\a}\convr x_1(u)$, $x_{\a}\convr x_2(u)$, and $x_2\ne x_1$.
Then, there exist two sequences $({\a}'_n)$ and $({\a}''_n)$ of indices with $\pm n(x_{\a}-x_1)\leqslant u$ 
and $\pm n(x_2-x_{\a})\leqslant u$ for ${\a}\geqslant{\a}'_n$ and ${\a}\geqslant{\a}''_n$. For each $n\in\mathbb{N}$ find 
${\a}'''_n\geqslant{\a}'_n,{\a}''_n$. The sequence $({\a}'''_n)$ satisfies
$\pm n((x_{\a}-x_1)+(x_2-x_{\a}))\leqslant 2u$ for ${\a}\geqslant{\a}'''_n$. Since $X$ is Abelian,
we conclude that $\pm n(x_2-x_1)\leqslant 2u$ for all $n\in\mathbb{N}$. Then $x_2=x_1$ because $X$ is almost Archime\-dean, a contradiction.
\end{proof}

\begin{definition}\label{Definition-5_3}
A subset $S$ of a POG $X$ is {\color{blue}{ru-closed}} if $S\ni x_{\a}\convr x(w)$ implies $x\in S$.
\end{definition}

\noindent
Let $X$ be a POG. We denote $N_X:=\bigl\{x\in X: (\exists\xi\in X_+)(\forall n\in\mathbb{N})\pm nx+\xi\in X_+\bigl\}$
and $D_X:=\{x\in X|(\exists\xi\in X_+)(\forall n\in\mathbb{N})\ nx+\xi\in X_+\}$.
It is straightforward to see that $X$ is almost Archimedean if and only if $N_X=\{0\}$.
Furthermore, $N_X$ is an order ideal in $X$ whenever $X$ is Abelian.  Cleary, $D_X\cap-D_X=N_X$.

\begin{proposition}\label{Proposition-5_1}
Let $X$ be a POG. Then $X$ is Archimedean if and only if $D_X=X_+$.
\end{proposition}

\begin{proof}
$(\Longrightarrow)$:\ 
As $X_+\subseteq D_X$, we have to show $D_X\subseteq X_+$.
Let $x\in D_X$. Then, for some $\xi\in X_+$, $\xi\geqslant -nx$ for every $n\in\mathbb{N}$.
Since $(X,X_+)$ is Archimedean, we infer $-x\leqslant 0$, or $x\in X_+$.

\smallskip
$(\Longleftarrow)$:\ 
Let $ny\leqslant\xi$ for some $\xi\in X_+$ and all $n\in\mathbb{N}$.
Then, $-ny+\xi\geqslant 0$ for all $n\in\mathbb{N}$, and hence $-y\in D_X$. Since $D_X=X_+$ then $y\leqslant 0$. Thus, $(X,X_+)$ is Archimedean.
\end{proof}

The following assertion is analogous to Assertion \ref{Assertion-3_1}.

\begin{assertion}\label{Assertion-5_1}
Let $X$ be a POG. Then $N_X=\{0\}^{(1)}_{ru}$.
\end{assertion}

\begin{proof}
Let $x\in X$. Then, $x\in N_X$ iff $\pm nx+\xi\in X_+$ for some $\xi\in X_+$ and all $n\in\mathbb{N}$ iff
$\pm n(0-x)\leqslant\xi$ for some $\xi\in X_+$ and all $n\in\mathbb{N}$ iff $x\in\{0\}^{(1)}_{ru}$.
\end{proof}

\medskip
The following definition is a POG version of Definition \ref{Definition-3_1}.

\begin{definition}\label{Definition-5_4}
Let $X$ be a POG and $\mathcal{C}^a_{Arch_G}(X)$ a category, whose objects are pairs $\left\langle Y,\phi\right\rangle$,
where $\phi:X\to Y$ is a positive group homomorphism to an almost Archimedean Abelian OG $Y$, and morphisms
$\left\langle Y_1,\phi_1\right\rangle\to\left\langle Y_2,\phi_2\right\rangle$ are positive group homomorphisms $\psi: Y_1\to Y_2$ satisfying
$\psi\circ\phi_1=\phi_2$. If $\left\langle Y_0,\phi_0\right\rangle$ is an initial object of $\mathcal{C}^a_{Arch}(X)$ 
then the OG $Y_0$ is said to be {\color{blue}{almost Archimedeanization}} of a POG $X$.
\end{definition}

Let us present a construction of almost Archimedeanization similar to what we did in Theorem \ref{Theorem-3_1}.
Consider the following sets in an Abelian POG $X$:\ $N_0(X)=\{0\}$ and, for every ordinal $\lambda>0$,  
$N_{\lambda}(X)=\bigl\{x\in X: [x]_{\bigcup_{{\g}<\lambda}N_{{\g}}(X)}\in N\bigl(X/_{\bigcup_{{\g}<\lambda}N_{{\g}}(X)}\bigl)\bigl\}$ 
(here, we use the straightforward fact that if $X$ an Abelian POG then each $N_{\lambda}(X)$ as well as $\bigcup_{{\g}<\lambda}N_{{\g}}(X)$
is an order ideal in $X$). Since $N_{\lambda_1}(X)\subseteq N_{\lambda_2}(X)$ for $\lambda_1\leqslant\lambda_2$ then 
$N_{\lambda+1}(X)=N_{\lambda}(X)$ for $\lambda\geqslant\card(X)$. Therefore, there exists
the first ordinal, denoted by $\lambda_X$, such that $N_{\lambda_X+1}(X)=N_{\lambda_X}(X)$.
We call $\lambda_X$ the {\color{blue}{almost Archimedeanization type}} of $X$.

\begin{lemma}\label{Lemma-5_2}
Let $T$ be a positive group homomorphism from an Abelian POG $(X,X_+)$ to an Abelian POG $(Y,Y_+)$. 
Then, $T(N_{\lambda}(X))\subseteq N_{\lambda}(Y)$ for every ordinal ${\g}$.
\end{lemma}

\begin{proof}
Trivially, $T(N_{0}(X))\subseteq N_{0}(Y)$.
Let $x\in N_{{\a}+1}(X)$. Then $\pm n[x]_{N_{{\a}}(X)}\leqslant [u]_{N_{{\a}}(X)}$ for some $u\in X_+$ and all $n\in\mathbb{N}$, 
and hence $\pm n[Tx]_{N_{{\a}}(Y)}\leqslant [Tu]_{N_{{\a}}(Y)}$ for $n\in\mathbb{N}$, which implies $Tx\in N_{{\a}+1}(Y)$. 
If ${\g}$ is a limit ordinal and $x\in N_{{\g}}(X)$ then $x\in N_{{\a}}(X)$ for some ${\a}<{\g}$. 
Thus, $\pm n[x]_{N_{{\a}}(X)}\leqslant [u]_{N_{{\a}}(X)}$ for some $u\in X_+$ and all $n\in\mathbb{N}$, and hence
$\pm n[Tx]_{N_{{\a}}(Y)}\leqslant [Tu]_{N_{{\a}}(Y)}$ for all $n\in\mathbb{N}$. It follows $Tx\in N_{{\a}+1}(Y)\subseteq N_{{\g}}(Y)$.
\end{proof}

\begin{theorem}\label{Theorem-5_1}
Every Abelian POG possesses a unique, up to an order isomorphism, almost Archime\-deanization.
\end{theorem}
 
\begin{proof}
Let $X$ be an Abelian POG. Then, $(X/N_{\lambda_X}(X),[X_+]_{N_{\lambda_X}(X)})$ is an almost Archimedean Abelian OG. 
Denote by $p_X$ the quotient map $X\to X/N_{\lambda_X}(X)$. Clearly, $p_X$ is a positive group homomorphism.
Let $Y$ be an almost Archimedean Abelian OG and $\phi:X\to Y$ a positive group homomorphism. By Lemma \ref{Lemma-5_2}, 
$\phi(N_{\lambda_X}(X))\subseteq N_{\lambda_X}(Y)$. Since $Y$ is almost Archimedean then $N_{\lambda_X}(Y)=\{0\}$.
Therefore, $N_{\lambda_X}(X)\subseteq\ker(\phi)$. So, a mapping $\tilde{\phi}:X/N_{\lambda_X}(X)\to Y$ 
such that $\tilde{\phi}([x]_{N_{\lambda_X}})=\phi(x)$ is a well defined positive group homomorphism satisfying $\tilde{\phi}\circ p_X=\phi$. 
Thus, $\left\langle(X/N_{\lambda_X}(X),[X_+]_{N_{\lambda_X}(X)}),p_X\right\rangle$ is an initial object of $\mathcal{C}^a_{Arch_G}(X)$,
and hence the OG $(X/N_{\lambda_X}(X),[X_+]_{N_{\lambda_X}(X)})$ is an {\em almost Archimedization} of a POG $X$.
\end{proof}


\end{document}